\documentclass[11pt]{amsart}
\usepackage{amsmath}
\usepackage[all]{xy}
\usepackage{amssymb}
\usepackage{mathrsfs}
\usepackage[noadjust]{cite}
\usepackage{hyperref}
\usepackage{amsfonts}
\usepackage{mathdots}
\usepackage{todonotes}
\usepackage{bm}
\usepackage{tikz}

\theoremstyle{plain}
\newtheorem{thm}{Theorem}[section]
\newtheorem{prop}[thm]{Proposition}
\newtheorem{lem}[thm]{Lemma}
\newtheorem{cor}[thm]{Corollary}

\theoremstyle{remark}
\newtheorem{rem}[thm]{Remark}

\theoremstyle{definition}
\newtheorem{defn}[thm]{Definition}

\theoremstyle{conjecture}

\newcommand{\Ku}{\mathcal{K}u}

\def\D{\mathrm{D}}

\def\H{\mathrm{H}}
\def\K{\mathrm{K}}
\def\S{\mathrm{S}}

\def\X{\mathrm{X}}
\def\Y{\mathrm{Y}}
\def\Z{\mathrm{Z}}

\def\Aut{\mathrm{Aut}}

\def\RHom{\mathrm{RHom}}

\def\Hom{\mathrm{Hom}}

\def\Jac{\mathrm{Jac}}

\def\gcd{\operatorname{gcd}}

\def\HH{\mathrm{HH}}

\def\HS{\mathrm{HS}}

\def\MF{\mathrm{MF}}

\def\Id{\mathrm{Id}}

\def\Tr{\mathrm{Tr}}

\def\Per{\mathrm{Per}}

\def\Perf{\mathrm{Perf}}

\title{Group action of Hochschild-Serre algebra and categorical reconstruction}
\author{Xun Lin}
\address{School of Science and Engineering, The Chinese University of Hong Kong,
Shenzhen, China}
\email{linxun@cuhk.edu.cn;lin-x18@tsinghua.org.cn}

\begin{document}
\maketitle
\begin{abstract}
  We study the action of the Serre functor of smooth proper dg categories at their Hochschild-Serre algebra, and the invariant sub-aglebra of the Serre functor. As applications, we prove some theorems of categorical Torelli. Namely,
  let $\Ku(\X)$ be the Kuznetsov component of degree $d$ smooth hypersurface in weighted projective space $\mathbb{P}(a_0,  a_1, \cdots, a_n)$, where the common maximal divisor $\gcd(d, \sum^{n}_{i=0}a_{i})=1$. We show the categorical Torelli for $\Ku(\X)$. We show that the $\mathbb{C}^{\ast}$ equivariant matrix factorization category associated with a quasi-homogeneous polynomial function $f$ that has an isolated singularity together with a twisted functor $\{1\}$ reconstructs $f$ up to an isomorphism. 
\end{abstract}
\section{Introduction}
Given a smooth proper variety $\X$, the bounded derived category of coherent sheaves $\D^{\mathrm{b}}(\X)$ is considered the noncommutative space of $\X$. We are interested in the $k$ linear $dg$ categories, and the smooth proper $dg$ categories as counterparts of smooth proper algebraic varieties. For a smooth proper dg category $\mathcal{A}$, we have a Serre functor $S_{\mathcal{A}}$, which plays a similar role as the canonical divisor. There is a natural invariant of $\mathcal{A}$, the Hochschild-Serre algebra $\HS(\mathcal{A})$, which is the non-commutative generalization of the canonical ring of algebraic varieties.
That is, if $\mathcal{A}=\D^{\mathrm{b}}(\X)$, the canonical ring $\bigoplus_{i\geq 0} \H^{0}(\X, \K^{\otimes i}_{\X})$ is a natural subring of $\HS(\D^{\mathrm{b}}(\X))$. 
The pair $(\mathcal{A}, \HS(\mathcal{A}))$ was shown to determine the isomorphism classes of certain hypersurface $\X$ when $\mathcal{A}\subset \D^{\mathrm{b}}(\X)$ is the residue subcategory \cite{lin2024serre}\cite{lin2024ivhskuznetsovcomponentscategorical}. 
\par
In this paper, we investigate a refined invariant $\HS(\mathcal{A})^{\rho(\S^{i}_{\mathcal{A}}[t])}$, which is an invariant sub-bi-graded algebra of $\HS(\mathcal{A})$ under the action of functor $\S^{i}_{\mathcal{A}}[t]$, see Theorem~\ref{invariant}. The refined invariant is a derived Morita invariant by Theorem~\ref{moritainvariant}. Let $\mathcal{A}=\D^{\mathrm{b}}(\X)$, and assume the canonical divisor $\mathrm{K}_{X}$ is nontrivial, the refined invariants of $\mathcal{A}$ are already interesting. It is expected that the refined invariant also has applications to the bi-rational geometry motivated by our main reconstruction Theorem~\ref{mainthm} below. 
\subsection*{Reconstruction}
Let $\X$ be a smooth hypersurface of degree $d$ in $\mathbb{P}(a_{0},a_{1},\cdots,a_{n})$. We write
$$\Ku(\X)=\langle \mathcal{O}_{\X}, \mathcal{O}_{\X}(1),\cdots,\mathcal{O}_{\X}(\sum^{n}_{i=0}a_{i}-d)\rangle^{\perp}\subset \D^{\mathrm{b}}(\X).$$
Assume $\operatorname{gcd}(d,\sum^{n}_{i=0}a_{i})=1$, our main result shows $(\Ku(\X), \HS(\Ku(\X))^{\rho(\S_{\Ku(\X)}[-n-1])})$ reconstructs the isomorphism classes of $\X$.
\begin{thm}(= Theorem~\ref{reconstruction2})\label{mainthm}
 Let $\X$ and $\Y$ be smooth degree $d$ hypersurface in $\mathbb{P}(a_{0},a_{1},\cdots,a_{n})$. Assume $\operatorname{gcd}(d,\sum^{n}_{i=0}a_{i})=1$. If there is a Fourier-Mukai equivalence $\Ku(\X)\cong \Ku(\Y)$, then $\X\cong \Y$.   
\end{thm}


\subsection*{Reconstruction with twisted functor}
  Let $\Y$ be a degree $d\leq n$ smooth hypersurface in $\mathbb{P}^{n}$. Define the rotation functor of $\Ku(\Y)$ as $O: E\mapsto \operatorname{Pr}(E\otimes^{\mathbb{L}} \mathcal{O}(1))$, which is considered as an induced polarization from $\mathcal{O}_{\Y}(1)$. The pair $(\Ku(\Y), O)$ reconstructs $Y$ by \cite{pirozhkov2024categorical}. Let $f\in k[x_{0},x_{1},\cdots,x_{n}]$ be a degree $d$ homogeneous polynomial such that $\Z(f)\subset \mathbb{P}(a_{0},a_{1},\cdots,a_{n})$ is smooth, here the weight $x_{i}$ is $a_{i}$. We assume $gcd(a_{0},a_{1},\cdots,a_{n})=1$. We have the equivariant matrix factorization category  $\MF(\mathbb{C}^{n+1},\mathbb{C}^{\ast},f)$ associated to $f$, which is equivalent to $\Ku(\Z(f))$ when $d\leq \sum^{n}_{i=0}a_{i}$, and the twisted functor $\{1\}$ of $\MF(\mathbb{C}^{n+1},\mathbb{C}^{\ast},f)$ is the rotation functor when $a_{0}=a_{1}=\cdots=a_{n}=1$, see \cite{BFK}. The following theorem shows the pair $(\MF(\mathbb{C}^{n+1},\mathbb{C}^{\ast},f), \{1\}$) reconstructs the polynomial $f$ up to a linear base change. 
\begin{thm}(= Theorem~\ref{matrixreconstruction2})\label{matrixreconstruction}
    Let $f_{1}$ and $f_{2}$ be two degree d homogeneous polynomial that defined smooth hypersurface in $\mathbb{P}(a_{0},a_{1},\cdots,a_{n})$. If there is a Morita equivalence of pairs $$(\MF(\mathbb{C}^{n+1},\mathbb{C}^{\ast},f_{1}),\{1\})\simeq (\MF(\mathbb{C}^{n+1},\mathbb{C}^{\ast},f_{2}),\{1\}),$$
   then $f_{1}$ and $f_{2}$ differ by a linear base change. 
\end{thm}
\begin{rem}
  Note that there is no extra numerical assumption of the degree $d$. If the weight $a_{0}=a_{1}=\cdots=a_{n}=1$, it was proved in \cite{lin2024serre}. 
\end{rem}
To prove Theorem~\ref{mainthm}, we study the bi-graded algebra
$\HS(\Ku(\X))^{S_{\Ku(\X)}[-n-1]}$. A detail analysis of the action of functor $\S_{\Ku(\X)}[-n-1]$ at $\HS(\Ku(\X))$ shows a decomposition of eigenspaces with eigenvalues $\langle \exp(\frac{2\pi i}{d}), \exp(2\cdot\frac{2\pi i}{d}),\cdots, \exp(d\cdot\frac{2\pi i}{d})\rangle$, see Theorem~\ref{eigenspacedecomposition}. The invariant subspace is nothing but the eigenspace for eigenvalue $1$, therefore we have a natural sub-ring,
$$\Jac(f_{\X})\subset \HS(\Ku(\X))^{\S_{\Ku(\X)}[-n-1]},$$
where $\Jac(f_{\X})$ is Jacobian ring of defined polynomial $f_{\X}$ of $\X$. The reconstruction follows from the Mather-Yau reconstruction \cite[Proposition 1.3]{Donagi1986}. To prove Theorem~\ref{matrixreconstruction}, we show the invariant $\HS((\MF(\mathbb{C}^{n+1},\mathbb{C}^{\ast},f))^{\{1\}}$ reconstructs the Jacobian ring of $f$, then the statement follows from the Mather-Yau reconstruction \cite[Proposition 1.3]{Donagi1986}.
\subsection*{Other applications}
According to \cite[Propositino 1.2]{perry2021hochschild}, $\S_{\Ku(\X)}[-n-1]$ acts trivially at the Hochschild cohomology. Combining Theorem~\ref{eigenspacedecomposition}, we show that the dimension of $\HH^{2}(\X)$ is the dimension of $\Jac(f_{\X})_{d}$, see Theorem~\ref{dimequal}. Another application is to give an another proof of \cite[Theorem 3.1]{linssz2024kunetsovfanoconjecture}, see Theorem~\ref{kernelgamma}. 

\section{Group action on Hochschild-Serre algebra}
Let $\mathcal{A}$ be a $k$ linear smooth proper $dg$ category. We write $\Per(\mathcal{A})$ as the $dg$ category of perfect right $\mathcal{A}$ modules, whose homotopic category $\Perf(\mathcal{A})$ is the triangulated category of perfect modules. We naturally identify right $\mathcal{A}^{op}$ module as left $\mathcal{A}$ module. An object $\K\in \Per(\mathcal{A}^{op}\otimes \mathcal{A})$ defined a dg functor $-\otimes_{\mathcal{A}} \K: \Per(\mathcal{A})\rightarrow \Per(\mathcal{A})$. According to \cite{shklyarov2007serre},  $\Perf(\mathcal{A})$ admits Serre functor, $\S_{\mathcal{A}}=-\otimes^{\mathbb{L}}_{\mathcal{A}}\mathcal{A}^{*}$, where $\mathcal{A}^{\ast}=\RHom(\mathcal{A},k)$.
The inverse Serre functor $\S^{-1}_{\mathcal{A}}=-\otimes^{\mathbb{L}}_{\mathcal{A}}\mathcal{A}^{\vee}$, where $\mathcal{A}^{\vee}=\RHom_{\Perf(\mathcal{A}^{op}\otimes\mathcal{A})}(\mathcal{A}, \mathcal{A}^{op}\otimes \mathcal{A})$. 
\begin{lem}\cite[Theorem 4.5]{shklyarov2007serre}
  We have isomorphism in $\Perf(\mathcal{A}^{op}\otimes \mathcal{A})$,
  $$\mathcal{A}\cong \mathcal{A}^{\vee}\otimes^{\mathbb{L}}_{\mathcal{A}}\mathcal{A}^{\ast} \cong \mathcal{A}^{\ast}\otimes^{\mathbb{L}}_{\mathcal{A}}\mathcal{A}^{\vee}$$
\end{lem}

From now on, we interpret $\Id_{\mathcal{A}}$, $\S^{-1}_{\mathcal{A}}$, and $\S_{\mathcal{A}}$ as bi-module $\mathcal{A}$, $\mathcal{A}^{\vee}$ and $\mathcal{A}^{*}$ respectively. We interpret $\Hom(\Id_{\mathcal{A}},\S^{m}_{\mathcal{A}})$ as the morphism of bi-modules in the derived category $\Perf(\mathcal{A}^{op}\otimes \mathcal{A})$. Note that we have natural identification of bi-modules $\S_{\mathcal{A}}^{m}\circ \S_{\mathcal{A}}^{n}=\S_{\mathcal{A}}^{m+n}$, $[n]\circ \S_{\mathcal{A}}=\S_{\mathcal{A}}\circ [n]$. 
\begin{defn}\cite{belmans2023hochschild}\cite{lin2024serre}
    The Hochschild-Serre algebra of $\mathcal{A}$ is a bi-graded algebra
    $$\HS(\mathcal{A}):=\bigoplus_{m,n}\Hom(\Id_{\mathcal{A}}, \S^{m}_{\mathcal{A}}[n])$$
\end{defn}

\subsection*{Commutators of Serre functor}
Let $F$ be a Morita auto-equivalence of $\mathcal{A}$ with inverse $F^{-1}$.
\begin{prop}\cite{belmans2023hochschild}
   There is an isomorphism of bi-modules $\sigma_{F}: F\circ S_{\mathcal{A}}\cong S_{\mathcal{A}}\circ F$. 
\end{prop}
\begin{proof}
 $F$ considered as bi-module defines an auto-equivalence $\Perf(\mathcal{A})\rightarrow \Perf(\mathcal{A})$. Then $\RHom_{\Perf(\mathcal{A})}(F,\mathcal{A})$ is the inverse of $F$ with natural bi-modules structure, which is denoted as $F^{-1}$.
 \begin{lem}
    The module $F\otimes_{k} F^{-1}\in \Perf(\mathcal{A}\otimes_{k} \mathcal{A}^{op}\otimes_{k} \mathcal{A}^{op}\otimes_{k}\mathcal{A})$ defined by 
    $$\Perf(\mathcal{A}^{op}\otimes \mathcal{A})\rightarrow  \Perf(\mathcal{A}^{op}\otimes \mathcal{A}),\ \ M \mapsto M\otimes^{\mathbb{L}}_{\mathcal{A}^{op}\otimes \mathcal{A}}F\otimes_{k} F^{-1}\cong F\otimes^{\mathbb{L}}_{\mathcal{A}} M\otimes^{\mathbb{L}}_{\mathcal{A}} F^{-1}$$
    maps $\mathcal{A}$ as bi-module to $\mathcal{A}$ itself. In particular, it defines an equivalence
 \end{lem}
 \begin{proof}
    The left multiplication defines isomorphism $$\mathcal{A}\rightarrow \RHom_{\Perf(\mathcal{A})}(F,F)\cong F\otimes^{\mathbb{L}}_{\mathcal{A}} F^{-1}\cong \mathcal{A}\otimes^{\mathbb{L}}_{\mathcal{A}^{op}\otimes \mathcal{A}}F\otimes_{k}F^{-1}.$$
    Similarly, the evaluation counit morphism is an isomorphism,
    $$F^{-1}\otimes^{\mathbb{L}}_{\mathcal{A}}F\cong \mathcal{A}.$$
   The inverse of $F\otimes_{k}F^{-1}$ is $F^{-1}\otimes_{k} F$, namely, for any $M\in \Perf(\mathcal{A}^{op}\otimes_{k}\mathcal{A})$, we have
   $$ F^{-1}\otimes^{\mathbb{L}}_{\mathcal{A}} F\otimes^{\mathbb{L}}_{\mathcal{A}} M\otimes^{\mathbb{L}}_{\mathcal{A}} F^{-1}\otimes^{\mathbb{L}}_{\mathcal{A}} F\cong \mathcal{A}\otimes^{\mathbb{L}}_{\mathcal{A}} M\otimes^{\mathbb{L}}_{\mathcal{A}} \mathcal{A}\cong M.$$
 \end{proof}
  Then by \cite[Lemma 1.2(ii)]{polishchuk2014lefschetz}, 
  \begin{align*}
\mathcal{A}^{\vee}\cong( \mathcal{A}\otimes^{\mathbb{L}}_{\mathcal{A}^{op}\otimes\mathcal{A}}F^{-1}\otimes_{k} F)^{\vee}
  \cong & F\otimes_{k} F^{-1}\otimes^{\mathbb{L}}_{\mathcal{A}\otimes \mathcal{A}^{op}}\mathcal{A}^{\vee}\\
  \cong & \mathcal{A}^{\vee}\otimes^{\mathbb{L}}_{\mathcal{A}^{op}\otimes\mathcal{A}}F\otimes_{k}F^{-1}.
  \end{align*}

In other words, $F\otimes_{k} F^{-1}$ maps $\mathcal{A}^{\vee}$ to $\mathcal{A}^{\vee}$. 
Therefore we have an isomorphism,
 $$F\circ \S^{-1}_{\mathcal{A}}\circ F^{-1}\cong \S^{-1}_{\mathcal{A}}.$$
Using $S_{\mathcal{A}}=S_{\mathcal{A}}^{-1}\circ \S_{\mathcal{A}}^{2}$, we also have 
$$F\circ \S_{\mathcal{A}}\circ F^{-1}\cong \S_{\mathcal{A}}.$$
\end{proof}
We call the isomorphism as commutator isomorphism for Serre functor. Then we have natural isomorphism for any integer $m$ and $n$ by induction 
$$\sigma_{F,m,n}: F\circ \S^{m}_{\mathcal{A}}[n]\cong \S^{m}_{\mathcal{A}}[n]\circ F,$$
where we naturally identify $[n]\circ F\cong F\circ [n]$,
\begin{lem}\label{associativecommutator}
   Let $F_{1}$ and $F_{2}$ be Morita equivalence, then there is a natural isomorphism of commutators $\sigma_{F_{1}\circ F_{2},m,n}\cong \sigma_{F_{1},m,n}\circ\sigma_{F_{2},m,n}$.
\end{lem}
\begin{proof}
   The isomorphism follows from the natural isomorphism 
   $$\xymatrix{(F_{1}\circ F_{2})\circ \S^{m}[n]\circ (F_{1}\circ F_{2})^{-1}\ar[r]^{\simeq}&F_{1}\circ (F_{2}\circ \S^{m}[n]\circ F_{2}^{-1})\circ F_{1}^{-1}}$$
\end{proof}
We write $\Aut_{M}(\mathcal{A})$ as the group of Morita auto-equivalence of $\mathcal{A}$. Given $F\in \Aut_{M}(\mathcal{A})$, we define an isomorphism of Hochschild-Serre algebra as
follows. For any $a\in \Hom(\Id,\S_{\mathcal{A}}^{m}[n])$, defined $F_{\ast}(a)\in \Hom(\Id,\S_{\mathcal{A}}^{m}[n])$ by the compositions
$$\xymatrix@C=1.7cm{\Id_{\mathcal{A}} \ar[r]^{\eta_{F}}&F\circ \Id_{\mathcal{A}}\circ F^{-1}\ar[r]^{a}&F\circ \S_{\mathcal{A}}^{m}[n]\circ F^{-1}\ar[r]^{\sigma_{F,m,n}}&S_{\mathcal{A}}^{m}[n]\circ F\circ F^{-1}\ar[r]_{\eta^{-1}_{F}}& \S_{\mathcal{A}}^{m}[n]}.$$
In the case of Hochschild homology, the definition $F_{\ast}$ coincide with the definition in \cite{caldararu2003mukai}\cite{polishchuk2014lefschetz}\cite{perry2021hochschild}.
\subsection*{Group homomorphism}
We write $\Aut(\HS(\mathcal{A}))$ as the auto-isomorphism of bi-graded algebra $\HS(\mathcal{A})$.
\begin{thm}
 We have a group homomorphism,
 $$\rho: \Aut_{M}(\mathcal{A})\longrightarrow \Aut(\HS(\mathcal{A})),\ \rho(F)=F_{\ast}.$$
 
\end{thm}
\begin{proof}
  First, $F_{\ast}: \HS(\mathcal{A})\rightarrow \HS(\mathcal{A})$ is an isomorphism of bi-graded algebra. Let $b\in \Hom(\Id,\S^{m_{1}}[n_{1}])$ and $a\in \Hom(\Id,\S^{m_{2}}[n_{2}])$. We have a commutative diagram, 
  $$\xymatrix{\Id\ar[r]\ar[d]&F\circ \Id\circ F^{-1}\ar[r]^{b}&F\circ \S^{m_{1}}[n_{1}]\circ F^{-1}\ar[d]\\
  F\circ \Id\circ F^{-1}\ar[d]^{b}&&F\circ \S^{m_{1}}[n_{1}]\circ \Id \circ F^{-1}\ar[d]\\
  F\circ \S^{m_{1}}[n_{1}]\circ \Id \circ F^{-1}\ar[d]^{a}&&F\circ \S^{m_{1}}[n_{1}]\circ F^{-1}\circ F\circ \Id\circ F^{-1}\ar[d]^{a}\\
  F\circ \S^{m_{1}}[n_{1}]\circ \S^{m_{2}}[n_{2}]\circ F^{-1}\ar[d]&&F\circ \S^{m_{1}}[n_{1}]\circ F^{-1}\circ F\circ \S^{m_{2}}[n_{2}]\circ F^{-1}\ar[d]\\
  \S^{m_{1}+m_{2}}[n_{1}+n_{2}]\circ F\circ F^{-1}\ar[rrdd]&&F\circ \S^{m_{1}}[n_{1}]\circ \Id \circ \S^{m_{2}}[n_{2}]\circ F^{-1}\ar[d]\\
  &&\S^{m_{1}+m_{2}}[n_{1}+n_{2}]\circ F\circ F^{-1}\ar[d]\\
  &&\S^{m_{1}+m_{2}}[n_{1}+n_{2}]}$$
  The left hand side composition is $F_{\ast}(a\cdot b)$, and the right hand side composition is $F_{\ast}(a)\cdot F_{\ast}(b)$, therefore $F_{\ast}(a\cdot b)=F_{\ast}(a)\cdot F_{\ast}(b)$. Next, $\rho$ is a group homomorphism by Lemma~\ref{associativecommutator}.
\end{proof}
    
For $F\in \Aut_{M}(\mathcal{A})$, we have an invariant sub-space $\HS(\mathcal{A})^{\rho(F)}\subset \HS(\mathcal{A})$.
\begin{cor}\label{invariant}
   $\HS(\mathcal{A})^{\rho(F)}$ is a sub-bi-graded algebra of $\HS(\mathcal{A})$.
\end{cor}
\begin{proof}
    For any $a,b \in \HS(\mathcal{A})^{\rho(F)}$, 
    $\rho(F)(a\cdot b)=\rho(F)(a)\cdot \rho(F)(b)=a\cdot b$. It implies $a\cdot b \in \HS(\mathcal{A})^{\rho(F)}$. Thus, $\HS(\mathcal{A})^{\rho(F)}$ is a sub-bi-graded algebra of $\HS(\mathcal{A})$. 
\end{proof}
Given $\X$ a smooth projective variety, a natural subject is a pair $(\X,\K_{\X})$, where $\K_{\X}$ is the canonical divisor. Similar with the philosophy, for a smooth proper $dg$ category $\mathcal{A}$, we have a pair $(\mathcal{A}, \S_{\mathcal{A}})$. We are interested in the invariant sub-algebra of Hochschild-Serre algebra under Serre functor, $\HS(\mathcal{A})^{\rho(\S^{\ast}_{\mathcal{A}}[\ast])}$.
 If $F: \mathcal{A}\cong \mathcal{B}$ is a Morita equivalence of smooth proper dg categories, we have a commutator isomorphism $F\circ \S_{\mathcal{A}}\circ F^{-1}\cong \S_{\mathcal{B}}$. Given any $a\in \Hom(\Id_{\mathcal{A}},\S_{\mathcal{A}}^{m}[n])$, we define $F_{\ast}(a)\in \Hom(\Id_{\mathcal{B}},\S^{m}_{\mathcal{B}}[n])$ to be 
 $$\xymatrix{\Id_{\mathcal{B}}\ar[r]&F\circ \Id_{\mathcal{A}}\circ F^{-1}\ar[r]^{a}&F\circ \S_{\mathcal{A}}^{m}[n]\circ F^{-1}\ar[r]&\S_{\mathcal{B}}^{m}[n]}.$$
 This defines an isomorphism of Hochschilld-Serre algebra\cite{belmans2023hochschild}\cite{lin2024serre},
 $$\HS(\mathcal{A})\cong \HS(\mathcal{B}).$$
\begin{thm}\label{moritainvariant}
  If $F: \mathcal{A}\cong \mathcal{B}$ is a Morita equivalence of smooth proper dg categories, then $F_{\ast}$ induces an isomorphism of bi-graded algebras for any integer $i$ and $t$,
  $$\HS(\mathcal{A})^{\rho(S^{i}_{\mathcal{A}}[t])}\cong \HS(\mathcal{B})^{\rho(S^{i}_{\mathcal{B}}[t])}.$$
\end{thm}
\begin{proof}
 For simplicity, we assume $i=t=0$. The proof for general $i$ and $t$ is similar.
 It suffices to prove $F_{\ast}(\HS(\mathcal{A})^{\rho(\S_{\mathcal{A}})})\subset \HS(\mathcal{B})^{\rho(\S_{\mathcal{B}})}$. Consider $a\in \Hom(\Id_{\mathcal{A}},\S_{\mathcal{A}}^{m}[n])^{\rho(\S_{\mathcal{A}})}$. To prove $F_{\ast}(a)\in \Hom(\Id_{\mathcal{B}},\S^{m}_{\mathcal{B}}[n])^{\rho(\S_{\mathcal{B}})}$, we break into several lemmas. 
 \begin{lem}\label{lem1}
  There is a commutative diagram,
$$\xymatrix{\S_{\mathcal{B}}\circ \Id_{\mathcal{B}}\circ \S^{-1}_{\mathcal{B}}\ar[r]\ar[d]&S_{\mathcal{B}}\circ F\circ \Id_{\mathcal{A}}\circ F^{-1}\circ \S^{-1}_{\mathcal{B}}\ar[d]\\
 \Id_{\mathcal{B}}\circ \S_{\mathcal{B}}\circ \S^{-1}_{\mathcal{B}}\ar[r]&F\circ \Id_{\mathcal{A}}\circ F^{-1}\circ \S_{\mathcal{B}}\circ \S^{-1}_{\mathcal{B}}}.$$
 \end{lem}
 \begin{proof}
  This follows from Lemma~\ref{associativecommutator}.   
 \end{proof}
 \begin{lem}\label{lem2}
  There is a commutative diagram,
  $$\xymatrix{\S_{\mathcal{B}}\circ F\circ \Id_{\mathcal{A}}\circ F^{-1}\circ \S^{-1}_{\mathcal{B}}\ar[r]^{a}\ar[d]&\S_{\mathcal{B}}\circ F\circ \S^{m}_{\mathcal{A}}[n]\circ F^{-1}\circ \S^{-1}_{\mathcal{B}}\ar[d]\\
 F\circ \Id_{\mathcal{A}}\circ F^{-1}\circ \S_{\mathcal{B}}\circ \S^{-1}_{\mathcal{B}}\ar[r]^{a}\ar[d]&F\circ \S^{m}_{\mathcal{A}}[n]\circ F^{-1}\circ \S_{\mathcal{B}}\circ \S^{-1}_{\mathcal{B}}\ar[d]\\
 F\circ \S^{m}_{\mathcal{A}}[n]\circ F^{-1}\circ \S_{\mathcal{B}}\circ \S^{-1}_{\mathcal{B}}\ar[r]&\S_{\mathcal{B}}^{m}[n]}.$$
 \end{lem}
 \begin{proof}
  The second commutative diagram is trivial. For the first commutative diagram, we use Lemma~\ref{associativecommutator} and the fact that $a$ is invariant under the action of $\S_{\mathcal{A}}$. Namely, we have a commutative diagram,
  $$\xymatrix{\S_{\mathcal{A}}\circ \Id_{\mathcal{A}}\ar[d]^{id\circ a}\ar[r]&\Id_{\mathcal{A}}\circ \S_{\mathcal{A}}\ar[d]^{a\circ id}\\
  \S_{\mathcal{A}}\circ \S^{m}_{\mathcal{A}}[n]\ar[r]&\S^{m}_{\mathcal{A}}[n]\circ \S_{\mathcal{A}}}.$$
 \end{proof}
\begin{lem}\label{lem3}
There is a commutative diagram
$$\xymatrix{\Id_{\mathcal{B}}\ar[r]\ar[rd]_{\eta_{1}}&\S_{\mathcal{B}}\circ \Id_{\mathcal{B}}\circ \S^{-1}_{\mathcal{B}}\ar[r]\ar[d]&\S_{\mathcal{B}}\circ F\circ \Id_{\mathcal{A}}\circ F^{-1}\circ \S^{-1}_{\mathcal{B}}\ar[r]^{a}\ar[d]&\S_{\mathcal{B}}\circ F\circ \S^{m}_{\mathcal{A}}[n]\circ F^{-1}\circ \S^{-1}_{\mathcal{B}}\ar[d]\\
 &\Id_{\mathcal{B}}\circ \S_{B}\circ \S^{-1}_{\mathcal{B}}\ar[r]_{\eta_{2}}&F\circ \Id_{\mathcal{A}}\circ F^{-1}\circ \S_{\mathcal{B}}\circ \S^{-1}_{\mathcal{B}}\ar[r]^{a}\ar[d]_{\eta_{3}}&F\circ \S^{m}_{\mathcal{A}}[n]\circ F^{-1}\circ \S_{\mathcal{B}}\circ \S^{-1}_{\mathcal{B}}\ar[d]\\
 &&F\circ \S^{m}_{\mathcal{A}}[n]\circ F^{-1}\circ \S_{\mathcal{B}}\circ \S^{-1}_{\mathcal{B}}\ar[r]^{\eta_{4}}&\S_{\mathcal{B}}^{m}[n]}.$$ 
 In particular, the composition $\eta_{4}\circ \eta_{3}\circ \eta_{2}\circ\eta_{1}=F_{\ast}(a)$, and the compositions of upper row and the fourth colum is $\rho(\S_{\mathcal{B}})(F_{\ast}(a))$.
\end{lem}
\begin{proof}
  The diagram is commutative because of Lemma~\ref{lem1} and Lemma~\ref{lem2}. $\eta_{4}\circ \eta_{3}\circ \eta_{2}\circ\eta_{1}=F_{\ast}(a)$ by definition. The last statement follows from definition.
\end{proof}
Finally, by Lemma~\ref{lem3}, $F_{\ast}(a)=\rho(\S_{\mathcal{B}})(F_{\ast}(a))$. In other words,  $F_{\ast}(\HS(\mathcal{A})^{\S_{\mathcal{A}}})\subset \HS(\mathcal{B})^{\mathcal{\S_{B}}}$. 
\end{proof}
\begin{rem}
 Take functors $F_{1}\in\Aut_{M}(\mathcal{A})$ and $F_{2}\in\Aut_{M}(\mathcal{B})$. If the Morita equivalence $F$ commutes with $F_{1}$ and $F_{2}$, then we have an isomorphism of bi-graded algebra,
 $$\bigoplus_{m,n}\Hom(\Id,F^{m}_{1}[n])^{\rho(F^{i}_{1}[t])}\cong \bigoplus_{m,n}\Hom(\Id,F_{1}^{m}[n])^{\rho(F^{i}_{2}[t])}.$$
 The proof is similar.
\end{rem}

\subsection*{Example}
Let $f$ be a degree $d$ quasi-homogeneous polynomial functions $f: \mathbb{C}^{n+1}\rightarrow \mathbb{C}$. Consider the dg category of matrix factorization category $\mathcal{A}=\MF(\mathbb{C}^{n+1},\mathbb{C}^{\ast},f)$. There is a twisted functor $\{1\}= -\otimes \mathcal{O}(1)$, and $\{1\}^{d}\cong[2]$. We write $\Delta(t)$ as the Fourier-Mukai Kernel for the functor $\{t\}=-\otimes \mathcal{O}(t)$.
\begin{prop}\cite[Theorem 5.39]{BFK}\label{extendedhochschild}
 If $f$ has an isolated singularity exactly at $0\in \mathbb{A}^{n+1}$, then 
\begin{align*}
 \Hom(\Delta,\Delta(t)[m])\cong& \bigoplus_{\{g\in \mu_{d}\ \mid\  \operatorname{rk}W_{g} \equiv m\pmod2\}}\mathrm{Jac}(f_{g})_{t-k_{g}+d\left(\frac{m-\operatorname{rk}W_{g}}{2}\right)}
 \end{align*}
\end{prop}
 $\{1\}^{d}\cong[2]$ implies the eigenvalues of the map $\rho(\{1\})$ at $\Hom(\Delta,\Delta(t)[m])$ is $1, \xi, \xi^{2}, \cdots, \xi^{d-1}$, where $\xi=\exp{\frac{2\pi i}{d}}$. 
 \begin{thm}\label{eigenspacedecomposition}
  The component $\Jac(f_{g})_{\bullet}$ in Proposition \ref{extendedhochschild} is the eigenspace of $\rho(\{1\})$ with eigenvalue $g$.
 \end{thm}   
\begin{proof}
We break the proof into several lemmas.
First, we have a natural isomorphism $\operatorname{Con}_{\S^{-1}}: \Hom(\Delta.\Delta(t)[m])\rightarrow \Hom(\S^{-1}\circ\Delta,\S^{-1}\circ\Delta(t)[m])$ defined as $\operatorname{Con}_{\S^{-1}}(a)=\S^{-1}\circ a$. For $b\in \Hom(\S^{-1}\circ \Delta,\S^{-1}\circ\Delta(t)[m])$, we define $\{1\}_{\ast}(b)$ as the composition,
$$\xymatrix{\S^{-1}\circ \Delta\ar[r]&\Delta(1)\circ \Delta(-1)\circ \S^{-1}\circ \Delta\ar[r]&\Delta(1)\circ \S^{-1}\circ \Delta\circ\Delta(-1)\ar[r]&\Delta(1)\circ \S^{-1}\circ \Delta(t)[m]\circ \Delta(-1)\ar[d]\\
&&&\S^{-1}\circ \Delta(t)[m]\circ \Delta(1)\circ \Delta(-1)\ar[d]\\
&&&\S^{-1}\circ \Delta(t)[m]}.$$
 \begin{lem}\label{lem4}
     We have a commutative diagram,
     $$\xymatrix@C=3cm{\Hom(\Delta,\Delta(t)[m])\ar[r]^{Con_{\S^{-1}}}\ar[d]^{\{1\}_{\ast}}& \Hom(S^{-1}\circ \Delta,\S^{-1}\circ \Delta(t)[m])\ar[d]^{\{1\}_{\ast}}\\
    \Hom(\Delta,\Delta(t)[m])\ar[r]^{Con_{\S^{-1}}}& \Hom(\S^{-1}\circ \Delta,\S^{-1}\circ \Delta(t)[m])}.$$ 
 \end{lem}

 \begin{proof}
    It is commutative by associativity of commutators, namely Lemma~\ref{associativecommutator}.  
 \end{proof}
 
 For any $dg$ category $\mathcal{B}$, we have trace functor $\Tr: \Per(\mathcal{B}\otimes \mathcal{B}^{op})\rightarrow \Per(k)$ defined by $\mathcal{B}\otimes_{\mathcal{B}\otimes\mathcal{B}^{op}} -$, and for any $M\in \Per(\mathcal{B}\otimes \mathcal{B})$, we have an isomorphism $\xymatrix{\Tr(M)\ar[r]^{ev}& \Hom(\S^{-1}_{\mathcal{B}}, M)}$, see for example \cite[Section 1.3]{polishchuk2014lefschetz}. Back to the matrix factorization category $\mathcal{A}$, the functor $\{1\}_{\ast}$ at $\Tr(S^{-1}\circ \Delta(t)[m])$ is defined by the composition below, see \cite[Section 2.1]{polishchuk2014lefschetz} for the definition in a more general setting,
 $$\xymatrix{\Tr(\S^{-1}\circ \Delta(t)[m])\ar[r]& \Tr(\Delta(1)\circ \Delta(-1)\circ \S^{-1}\circ \Delta(t)[m])\ar[r]&\Tr(\Delta(1)\circ \S^{-1}\circ \Delta(t)[m]\circ \Delta(-1))\ar[d]\\
 && \Tr(S^{-1}\circ \Delta(t)[m]\circ \Delta(1)\circ \Delta(-1))\ar[d]\\
 &&\Tr(\S^{-1}\circ \Delta(t)[m])}.$$
 \begin{lem}\label{lem5}
     There is a commutative diagram,
     $$\xymatrix@C=3cm{\Tr(\S^{-1}\circ \Delta(t)[m])\ar[r]^{ev}\ar[d]^{\{1\}_{\ast}}& \Hom(\S^{-1}\circ \Delta,\S^{-1}\circ \Delta(t)[m])\ar[d]^{\{1\}_{\ast}}\\
    \Tr(\S^{-1}\circ \Delta(t)[m])\ar[r]^{ev}& \Hom(\S^{-1}\circ \Delta,\S^{-1}\circ \Delta(t)[m])}.$$ 
 \end{lem}

 \begin{proof}
     This follows by \cite[Proposition 2.10]{polishchuk2014lefschetz}.
 \end{proof}
 
  \begin{lem}\label{lem6}
     We have an isomorphism 
     $$\Tr(\S^{-1}\circ \Delta(t)[m])\cong \bigoplus_{\{g\in \mu_{d}\ \mid\  \operatorname{rk}W_{g} \equiv m\pmod2\}}\mathrm{Jac}(f_{g})_{t-k_{g}+d\left(\frac{m-\operatorname{rk}W_{g}}{2}\right)}.$$
     The component $\Jac(f_{g})_{\bullet}$ is the eigenspace with eigenvalue $g$ under the action of $\{1\}$. Namely, for $a_{g}\in \Jac(f_{g})$, $\{1\}_{\ast}(a_{g})=g\cdot a_{g}$. 
 \end{lem}

 \begin{proof}
     For the case of Hochschild homology, this was proved in \cite[Theorem 2.6.1]{polishchuk2016matrix}. The general case is similar.
 \end{proof}
 
 \begin{lem}\label{lem7} The compositions in the isomorphisms below 
  $$\Hom(\Delta,\Delta(t)[m])\cong \Tr(\S^{-1}\circ \Delta(t)[m])\cong \bigoplus_{\{g\in \mu_{d}\ \mid\  \operatorname{rk}W_{g} \equiv m\pmod2\}}\mathrm{Jac}(f_{g})_{t-k_{g}+d\left(\frac{m-\operatorname{rk}W_{g}}{2}\right)}$$
    is the isomorphism in Proposition~\ref{extendedhochschild}. 
 \end{lem}
 
 \begin{proof}
   This is because both decompositions are from the decomposition of diagonal factorizations $\Delta$. The case for Hochschild homology is obtained in \cite[Theorem 6.5]{BFK}, the general case is similar.
 \end{proof}
 Now, by Lemma~\ref{lem6} and Lemma~\ref{lem7}, the statement follows. 
\end{proof}

\begin{thm}\label{subringjacobian}
If $\gcd(d, \sum^{n}_{i=0}a_{i})=1$, then there are integers $m_{0}$, $n_{0}$, and an isomorphism of algebra,
$$\Jac(f)\cong \bigoplus^{N}_{t}\Hom(\Id,\S^{m_{0}t}[n_{0}t])^{\S_{\mathcal{A}}[-n-1]}\subset \HS(\MF(\mathbb{C}^{n+1},\mathbb{C}^{\ast},f))^{\S_{\mathcal{A}}[-n-1]}.$$
Here $N$ is the maximal integer with $\Jac(f)_{N}\neq 0$.
\end{thm}

\begin{proof}
  First, according to \cite[Proposition 2.2]{lin2024ivhskuznetsovcomponentscategorical}, $\S_{\mathcal{A}}[-n-1]\cong \{-(\sum^{n}_{i=0}a_{i})\}[n+1][-n-1]\cong \{-(\sum^{n}_{i=0}a_{i})\}$. Since $\operatorname{gcd}(d,\sum^{n}_{i=0}a_{i})=1$, for any $g\neq 1 \in \mu_{d}$, $g^{-\sum^{n}_{i=0}a_{i}}\neq 1$. Therefore, $\Hom(\Delta,\Delta(t))^{S_{\mathcal{A}}[-n-1]}\cong \Jac(f)_{t}$ by Theorem~\ref{eigenspacedecomposition}. Next, similar with the method in \cite{lin2024serre}, we have some integer $m_{0}$, $n_{0}$ such that $\Hom(\Id,\S^{m_{0}t}[n_{0}t])\cong \Hom(\Delta,\Delta(t))\cong \Jac(f)_{t}$.
  Thus, we have an isomorphism of algebra,
  $$\Jac(f)\cong \bigoplus^{N}_{t=0}\Hom(\Id,\S^{m_{0}t}[n_{0}t])^{\S_{\mathcal{A}}[-n-1]}\subset \HS(\mathcal{A})^{\S_{\mathcal{A}}[-n-1]}.$$
\end{proof}
\section{Applications}
In this section, let $\X$ be a degree $d$ smooth hypersurface in $\mathbb{P}(a_{0},a_{1},\cdots,a_{n})$ defined by polynomial $f$. Assume $\operatorname{gcd}(d,\sum^{n}_{i=0}a_{i})=1$. We compute the dimension of $\HH^{2}(\Ku(\X))$, and prove the categorical reconstruction theorem for $\Ku(\X)$. We prove Theorem~\ref{matrixreconstruction}(= Theorem~\ref{matrixreconstruction2}), and give an another proof of \cite[Theorem 3.1]{linssz2024kunetsovfanoconjecture}, see Theorem~\ref{kernelgamma} below.


 \begin{thm}\label{dimequal}
 We have $\operatorname{dim} \HH^{2}(\Ku(\X))
=\operatorname{dim}\Jac(f)_{d}$. \end{thm}
 
 \begin{proof}
According to \cite[Proposition 1.2]{perry2021hochschild}, the functor $\S_{\mathcal{A}}[-n-1]$ acts trivially at $\HH^{2}(\Ku(\X))$. By Theorem~\ref{eigenspacedecomposition}, the corresponding decomposition of $\HH^{2}(\Ku(\X))\cong \Hom(\Delta,\Delta[2])\cong \Hom(\Delta,\Delta(d))$ has only one component $\Jac(f)_{d}$. Therefore, $\operatorname{dim} \HH^{2}(\Ku(\X))=\operatorname{dim}\Jac(f)_{d}$. 
 \end{proof}
 
 

\begin{thm}\label{reconstruction2}
Let $\X$ and $\Y$ be smooth degree $d$ hypersurface in $\mathbb{P}(a_{0},a_{1},\cdots,a_{n})$. Assume $\operatorname{gcd}(d,\sum^{n}_{i=0}a_{i})=1$. If there is a Fourier-Mukai equivalence $\Ku(\X)\cong \Ku(\Y)$, then $\X\cong \Y$.   
\end{thm}

\begin{proof}
  Assume $\X$ and $\Y$ are the degree $d$ hypersurfaces defined by polynomials $f_{\X}$ and $f_{\Y}$ respectively. Suppose there is a Fourier-Mukai equivalence $\Ku(\X)\cong \Ku(\Y)$, then according to Theorem~\ref{moritainvariant}, we have a natural isomorphism of bi-graded algebra,
  $$\HS(\Ku(\X))^{\S_{\Ku(\X)}[-n-1]}\cong \HS(\Ku(\Y))^{\S_{\Ku(\Y)}[-n-1]}.$$

  According to \cite[Theorem 6.13]{BFK}, we have a Morita equivalence,
   $$\Ku(\X)\cong \MF(\mathbb{C}^{n+1},\mathbb{C}^{\ast}, f_{X}).$$
Therefore, according to Theorem~\ref{moritainvariant} again, we have an isomorphism of bi-graded algebra,

$$\HS(\MF(\mathbb{C}^{n+1},\mathbb{C}^{\ast}, f_{\X}))^{\S_{f_{\X}}[-n-1]}\cong \HS(\MF(\mathbb{C}^{n+1},\mathbb{C}^{\ast}, f_{\Y}))^{\S_{f_{\Y}}[-n-1]},$$
where $\S_{f_{\X}}$ and $\S_{f_{\Y}}$ are the Serre functors of $\MF(\mathbb{C}^{n+1},\mathbb{C}^{\ast}, f_{\X})$ and $\MF(\mathbb{C}^{n+1},\mathbb{C}^{\ast}, f_{\Y})$ respectively. It induces an isomorphism of graded algebra via an embedding of grading $t\mapsto (m_{0}t,n_{0}t)$,
$$\bigoplus^{N}_{t=0}\Hom(\Id,\S_{f_{\X}}^{m_{0}t}[n_{0}t])^{\S_{f_{\X}}[-n-1]}\cong \bigoplus^{N}_{t=0}\Hom(\Id,\S_{f_{\Y}}^{m_{0}t}[n_{0}t])^{\S_{f_{\Y}}[-n-1]}$$
According to Theorem~\ref{subringjacobian}, we have an isomorphism of Jacobian algebra,
$$\Jac(f_{\X})\cong \Jac(f_{\Y}).$$
According to Mather--Yau reconstruction \cite[Proposition 1.3]{Donagi1986}, $\X\cong \Y$.
\end{proof}
\begin{thm}\label{matrixreconstruction2}
    Let $f_{1}$ and $f_{2}$ be two degree d homogeneous polynomial that defined smooth hypersurface in $\mathbb{P}(a_{0},a_{1},\cdots,a_{n})$. If there is a Morita equivalence of pairs $$(\MF(\mathbb{C}^{n+1},\mathbb{C}^{\ast},f_{1}),\{1\})\simeq (\MF(\mathbb{C}^{n+1},\mathbb{C}^{\ast},f_{2}),\{1\}),$$
   then $f_{1}$ and $f_{2}$ differ by a linear base change. 
\end{thm}

\begin{proof}
  We write $\{1\}_{f_{i}}$ as the twisted functor of $\MF(\mathbb{C}^{n+1},\mathbb{C}^{\ast},f_{i})$. According to Theorem~\ref{moritainvariant}, we have an isomorphism of bi-graded algebra,
  $$\xymatrix{\bigoplus^{N}_{i=0}\Hom(\Id,\{1\}^{i}_{f_{1}})^{\{1\}_{f_{1}}}\ar[r]^{\simeq}&\bigoplus^{N}_{i=0}\Hom(\Id,\{1\}^{i}_{f_{2}})^{\{1\}_{f_{2}}}}.$$
  $N$ is the maximal number with $\Jac(f_{1})\neq 0$. Then by Theorem~\ref{eigenspacedecomposition}, we have an isomorphism $\Jac(f_{1})\cong \Jac(f_{2})$. Thus, by \cite[Proposition 1.3]{Donagi1986}, $f_{1}$ and $f_{2}$ differ by a linear base change.
\end{proof}
Let $Y$ be a degree $4$ smooth hypersurface in $\mathbb{P}(1,1,1,1,2)$, and $\omega$ be the defined polynomial of $\Y$. Denote the residual category of $\Y$ as
$$\Ku(\Y)=\langle \mathcal{O}_{\        Y},\mathcal{O}_{\Y}(1)\rangle^{\perp}.$$
We have a natural map $\gamma_{\Y}$ defined by the action of Hochschild cohomology at Hochschild homology,
$$\xymatrix{\gamma_{\Y}:\HH^{2}(\Ku(\Y))\ar[r] & \Hom(\HH_{-1}(\Ku(\Y)),\Ku(\Y))\\}.$$
We give an another proof of the fact that the dimension of Kernel of $\gamma_{\Y}$ is one \cite[Theorem 3.1]{linssz2024kunetsovfanoconjecture}.
\begin{thm}\cite[Theorem 3.1]{linssz2024kunetsovfanoconjecture}\label{kernelgamma}
   The dimension of the kernel of $\gamma_{\Y}$ is one.
\end{thm}
\begin{proof}
 Similarly with \cite{linssz2024kunetsovfanoconjecture}, we identify $\Ku(\Y)$ with the category of matrix factorization $\MF(\mathbb{C}^{5},\mathbb{C}^{\ast},f)$. According to the computations in the proof of \cite[Theorem 3.1]{linssz2024kunetsovfanoconjecture}, we have 
  \begin{align*}
      \HH^{2}(\Ku(\Y))\cong& \Jac(\omega)_{4}\oplus \Jac(\omega\vert_{-1})_{0}\cong \Jac(\omega)_{4}\oplus \mathbb{C}.\\
      \HH_{-1}(\Ku(\Y))\cong & \Jac(\omega)_{2}.\\
      \HH_{1}(\Ku(\Y))\cong & \Jac(\omega)_{6}.
  \end{align*}
Consider any $\alpha\in\HH^{2}(\Ku(\Y))$ that project to $0\in \Jac(\omega)_{4}$. Let $\mathrm{m}$ be the map defined by multiplication of Jacobian ring,
$$\mathrm{m}: \Jac(\omega)_{4}\rightarrow \Hom(\Jac(\omega)_{2},\Jac(\omega)_{10}).$$
According to \cite[Theorem 1.6]{linssz2024kunetsovfanoconjecture}, we have a commutative diagram,
$$\xymatrix@C=3cm{\HH^{2}(\Ku(\Y))\ar[r]^{\gamma}&\Hom(\HH_{-1}(\Ku(\Y)),\HH_{1}(\Ku(\Y)))\ar[dd]^{\simeq}\\
\bigcup&\\
\Jac(\omega)_{4}\ar[r]^{m}& \Hom(\Jac(\omega)_{2},\Jac(\omega)_{6})}$$
Since $\mathrm{m}$ is injective, to show that the dimension of kernel of $\gamma_{\Y}$ is one is equivalent to show $\gamma(\alpha)=0$. Let $\beta\in \HH_{-1}(\Ku(\Y))$. Consider the twisted functor $\{1\}$, we have a commutative diagram,
  $$\xymatrix@C=3cm{\HH^{2}(\Ku(\Y))\ar[r]^{\gamma_{\Y}}\ar[d]^{\{1\}_{\ast}}&\Hom(\HH_{-1}(\Ku(\Y)),\HH_{1}(\Ku(\Y)))\ar[d]^{\{1\}_{\ast}}\\
  \HH^{2}(\Ku(\Y))\ar[r]^{\gamma_{\Y}}&\Hom(\HH_{-1}(\Ku(\Y)),\HH_{1}(\Ku(\Y)))\\}$$
Therefore, we have,
$$\{1\}_{\ast}(\gamma(\alpha)(\beta))=\gamma(\{1\}_{\ast}(\alpha))(\{1\}_{\ast}\beta).$$
According to Theorem~\ref{eigenspacedecomposition}, $\{1\}_{\ast}(\alpha)=-\alpha$, $\{1\}_{\ast}\beta=\beta$, and $\{1\}_{\ast}(\gamma(\alpha)(\beta))=\gamma(\alpha)(\beta)$.
We have,
$$\gamma(\alpha)(\beta)=-\gamma(\alpha)(\beta).$$
Therefore, $\gamma(\alpha)(\beta)=0$ for any $\beta\in \HH_{-1}(\Ku(\Z))$, and hence $\gamma(\alpha)=0$.
\end{proof}

\bibliographystyle{alpha}
{\small{\bibliography{reference}}}

\end{document}